\documentclass[11pt,reqno]{amsart}
\usepackage{graphicx}
\usepackage{verbatim}
\usepackage{textcomp}
\usepackage{amssymb}
\usepackage{cite}
\usepackage{amsmath}
\usepackage{latexsym}
\usepackage{amscd}
\usepackage{amsthm}
\usepackage{mathrsfs}
\usepackage{bm}
\usepackage{url}
\usepackage{hyperref}
\usepackage{bookmark}
\allowdisplaybreaks[3]
\newtheorem{thm}{Theorem}[section]
\newtheorem*{thma}{Theorem A}
\newtheorem{corr}[thm]{Corollary}
\newtheorem{lem}[thm]{Lemma}
\newtheorem{prop}[thm]{Proposition}

\theoremstyle{definition}
\newtheorem{defn}{Definition}[section]

\theoremstyle{remark}
\newtheorem{rem}{Remark}[section]
\numberwithin{equation}{section}
\begin{document}
\title[The $L_p$ Polar bodies of shadow system and related inequalities]
{The $L_p$ Polar bodies of shadow system and related inequalities}

\author{Lujun Guo}

\author{Hanxiao Wang}

\address{College of Mathematics and Information Science,
Henan Normal University, Xinxiang 453007, P.R. China}

\email{lujunguo0301@163.com(L. Guo)}
\email{lujunguo@htu.edu.cn(L. Guo)}
\email{2069384924@qq.com(H. Wang) }
\thanks{The research of authors is supported by NSFC (No. 12126319).}

\begin{abstract}
The $L_p$ versions of the support function and polar body are introduced by Berndtsson, Mastrantonis and Rubinstein in \cite{Berndtsson-Mastrantonis-Rubinstein-2023} recently.
In this paper, we prove that the $L_p$-support function of the shadow system $K_t$ introduced by Rogers and Shephard in \cite{rogers-1958-02,shephard-1964} is convex and the volume of the section of $L_p$ polar bodies of $K_t$ is $\frac{1}{n}$-concave with respect to parameter $t$, and obtain some related inequalities. Finally, we present the reverse Rogers-Shephard  type inequality for $L_p$-polar bodies.
\end{abstract}



\maketitle

\section{Introduction}

A convex body is a subset $K\subset \mathbb{R}^n$ which is convex, compact, and has non-empty interior. The support function $h_K$ of a convex body $K$ is defined by
$$h_K(x)=\max\{\langle x,y\rangle:y\in K\},\ \ \ \ x\in\mathbb{R}^n$$
where $\langle \cdot, \cdot\rangle$ denotes the standard inner product. By $int(K)$ we denote the interior of $K$. If $K$ is a convex body in $\mathbb{R}^n$ and $z\in int(K)$, the polar body $K^z$ of $K$ with the center of polarity $z$ is defined by
\begin{align}\label{polar body of K}
K^z=\{y\in\mathbb{R}^n:\langle y-z,x-z\rangle \leq1 \ \ \ \text{for\ all}\ \ x\in K\}.
\end{align}
If the center of polarity is taken to be the origin, we denote by $K^o$ the polarity of $K$, thus $K^z=(K-z)^o+z$. For $z\in int(K)$, $(K^z)^z=K$ (see \cite{meyer-pajor-1990} or \cite[Theorem 1.6.1]{schneider-2014}).
In \cite{Berndtsson-Mastrantonis-Rubinstein-2023},  Berndtsson, Mastrantonis and Rubinstein introduced $L_p$ versions of the support function $h_K$ and polar body $K^o$ of $K$ as follows.

\begin{defn}
For $p\in(0,+\infty]$ and $K\subset \mathbb{R}^n$, the $L_p$ support function $h_{p,K}$ of $K$ is defined by
\begin{align}\label{p-support-function}
h_{p,K}(y):=\log\Big(\int_{K}e^{p\langle x,y\rangle}\frac{dx}{|K|}\Big)^{\frac{1}{p}},\ \ \ y\in \mathbb{R}^{n}.
\end{align}
where $|K|$ is the volume of $K$ and the $L_p$ polar body $K^{o,p}$ of $K$ is defined by
\begin{align}\label{p-polar body}
K^{o,p}:=\{y\in \mathbb{R}^{n}:\|y\|_{K^{o,p}}\leq1\},
\end{align}
where
\begin{align}\label{norm of polar body of K}
\|y\|_{K^{o,p}}:=\Big(\frac{1}{(n-1)!}\int^{\infty}_{0}r^{n-1}e^{-h_{p,K}(ry)}dr\Big)^{-\frac{1}{n}}.
\end{align}
\end{defn}
From (\ref{norm of polar body of K}) and the formula in \cite[P276, Remark 5.1.2]{schneider-2014}, the volume of $K^{o,p}$ can be represented in the form
\begin{align}\label{volume of p-polar body}
|K^{o,p}|=\frac{1}{n}\int_{S^{n-1}}\|u\|_{K^{o,p}}^{-n}du=\frac{1}{n!}\int_{\mathbb{R}^n}e^{-h_{p,K}(x)}dx,
\end{align}
where in the first integral the integration is with respect to Lebesgue measure
on the unit sphere $S^{n-1}$ and in the second integral the integration is with respect to Lebesgue
measure on $\mathbb{R}^n$.
It is obvious that $h_K$ and $K^o$ are $L_\infty$ versions of $h_{p,K}$ and $K^{o,p}$, i.e.,
$$h_K=h_{\infty,K}\ \ \ \ \text{and}\ \ \ \  K^o=K^{o,\infty}.$$
Berndtsson et al also introduced $L_p$ Mahler volumes $\mathcal{M}_p(K)$ of a convex body $K$ in \cite{Berndtsson-Mastrantonis-Rubinstein-2023},
\begin{align}\label{p-mahler volume}
\mathcal{M}_p(K):=|K|\int_{\mathbb{R}^n}e^{-h_{p,K}(x)}dx,
\end{align}
and the authors also introduced $L_p$ Mahler conjectures and established many important and interesting results. One of their main results
is the following $L_p$ Blaschke-Santal\'{o}'s inequality.

\begin{thma}\cite{Berndtsson-Mastrantonis-Rubinstein-2023}\label{theorem-a}
Let $p\in(0,+\infty]$. For a symmetric convex body $K\subset \mathbb{R}^n$,
$$\mathcal{M}_p(K)\leq \mathcal{M}_p(B_2^n),$$
where $B_2^n=\{x\in\mathbb{R}^n:|x_1|^2+\cdots+|x_n|^2\leq1\}$.
\end{thma}
By taking $p\rightarrow +\infty$, one recovers the classical Blaschke-Santal\'{o}'s inequality:
$$\mathcal{M}(K)\leq \mathcal{M}(B_2^n).$$
For more reference of Blaschke-Santal\'{o}'s inequality, we refer to \cite{LYZ-2000,lutwak-zhang-1997,meyer-1991,meyer-pajor-1989,meyer-pajor-1990} and references therein.
 The idea used in the proof of Theorem A is standard: for $u\in S^{n-1}$, the Steiner symmetrization $S_uK$ of $K$ with respect to a hyperplane $u^\perp:=\{x\in\mathbb{R}^n:\langle x,u\rangle=0\}$ increases the volume of the $L_p$ polar body, i.e.,
 $$|(S_uK)^{o,p}|\geq |K^{o,p}|.$$

The notion of shadow system was introduced by Rogers and Shephard \cite{rogers-1958-02} and Shephard \cite{shephard-1964}. A shadow system along the direction $v\in S^{n-1}$ is a family of convex sets $K_t\subset\mathbb{R}^n$,
\begin{align}\label{shadow-system}
K_t=conv\{x+\alpha(x)tv:x\in K\subset\mathbb{R}^n\},
\end{align}
where \emph{conv} denotes convex hull, $K$ is an arbitrary bounded set of points, $\alpha(x)$ is bounded function on $K$ and $t\in\mathbb{R}$.
The function $\alpha(x)$ is called the speed function of the shadow system.
A parallel chord movement along the direction $v\in S^{n-1}$ is a particular type of shadow system defined by
\begin{align}\label{parallel chord movement}
K_t=\{x+\beta(x|v^\perp)tv:x\in K\},
\end{align}
where $K\subset\mathbb{R}^n$ is a convex body and $\beta(\cdot)$ is a continuous real function on the projection $K|v^\perp$
 of $K$ onto $v^\perp.$ The speed function $\beta(\cdot)$ has to be given in such a way that at time $t$ the union $K_t$ of moving chords is convex.
For the use of special shadow system, we refer to Bianchini and Colesanti \cite{Bianchini-2008}, Campi and Gronchi \cite{campi-2002,campi-2002-02,campi-2006,campi-2006-02}, Fradelizi, Meyer and Zvavitch \cite{fradelizi-meyer-2012} and Meyer and Reiser \cite{meyer-reisner-2006, meyer-reisner-2011} and references therein.

For any $v\in S^{n-1}$ and $s\in\mathbb{R}$, the projection of $K\cap(v^\perp+sv)$ of a convex body $K\subset\mathbb{R}^n$ onto $v^\perp$ is denoted by $K(s)$, i.e.,
$$K(s)=\{x':x'+sv\in K\}.$$

In this paper, we will prove that the $L_p$-support function of the shadow system $K_t$  is convex and the slice volume of $L_p$ polar body $K_t^{o,p}$ of a parallel chord movement $K_t$ is $\frac{1}{n}$-concave with respect to parameter $t$.

\begin{thm}\label{introduction-thm-p-support function is convex1}
 Let $p\in(0,\infty],v\in S^{n-1}$ and $K_t$ be a parallel movement of a convex body $K \subset \mathbb{R}^{n}$ along the direction $v$. Then the function $t\mapsto h_{p,K_t}$ is convex.
\end{thm}

\begin{thm}\label{thm-02}
Let $p\in(0,+\infty]$, $v\in S^{n-1}$ and $K_t$ be a parallel chord movement of a convex body $K$ along the direction $v$. Then for all $s, t_i\in\mathbb{R}$, $i=1,2$
\begin{align}\label{introduction-slice vulume inequality for p polar body of ti}
|K_{\frac{t_1+t_2}{2}}^{o,p}(s)|^{\frac{1}{n-1}}\geq \frac{1}{2}|K_{t_1}^{o,p}(s)|^{\frac{1}{n-1}}+\frac{1}{2}|K_{t_2}^{o,p}(s)|^{\frac{1}{n-1}}.
\end{align}
\end{thm}

Let $K$ be a convex body. For any $v\in S^{n-1}$, we have
\begin{align}\label{body-K}
K=\{x'+sv:x'\in K|v^\perp,  g_v(x')\leq s\leq f_v(x')\},
\end{align}
where $g_v(x')$ and $-f_v(x')$ are convex functions on $K|v^\perp$. If we take $\beta(\cdot)=-\big(g_v(\cdot)+ f_v(\cdot)\big)$, then $K_0=K$, $K_1=K^v$ is the reflection of $K$ in the hyperplane $v^\perp$, and $K_{\frac{1}{2}}=S_vK$ is the Steiner symmetrization of $K$ with respect to $v^\perp$.
Taking $t_1=0$, $t_2=1$ and $\beta(\cdot)=-\big(g_v(\cdot)+ f_v(\cdot)\big)$ in Theorem \ref{thm-02}, we have the following corollary.

\begin{corr}
 Let $K$ be a symmetric convex body. Then for any $v\in S^{n-1}$,
$$|(S_vK)^{o,p}|\geq |K^{o,p}|.$$
\end{corr}

The Minkowski sum of two convex bodies $K,L\subset \mathbb{R}^{n}$ is the convex body defined by
\begin{align}\label{sum-K  L}
K+L:=\{x+y\in\mathbb{R}^{n}:x\in K ,y\in L\}
\end{align}
and $K-L:=K+(-L)$. It was proved in \cite{Alonso-2016} (see also \cite{Artstein--2015}) that for any two convex bodies $K,L \subseteq \mathbb{R}^{n}$ with $o\in int K\cap int L$
\begin{align}\label{classical Rogers-Shephard inequality}
|(K\cap L)^o||(K-L)^o|\leq|K^o||L^o|,
\end{align}
and in \cite{Alonso-2016} Alonso at al also point out that (\ref{classical Rogers-Shephard inequality}) only converges to equality
when we consider sequences of sets $\{(K_n^o)_n\}$ and $\{(L_n^o)_n\}$ converging to simplices with $o$ in
one of the vertices and the same outer normal vectors at the facets that pass through the origin.

A reverse inequality was proved by Alonso in \cite{Alonso-2019} as follows:
Let $K,L \subseteq \mathbb{R}^{n}$ be convex bodies such that $K^o$ and $L^o$ have opposite barycenters.
\begin{align}\label{reverse classical Rogers-Shephard inequality}
|(K\cap L)^o||(K-L)^o|\geq\frac{(n!)^2}{(2n)!}|K^o||L^o|
\end{align}

For more references on Rogers-Shephard inequalities and its reverse, we refer to\cite{Alonso-2013,Alonso-2016,Alonso-2019,Artstein--2015,milman-2000,rogers-1957,rogers-1958} and reference therein.
In Section $3$ , we prove the following volumes inequality which is a reverse  Rogers-Shephard type inequality about $L_p$ polar bodies.
\begin{thm}\label{thm-reverse--R-S-type-inequality section1}
Let $p\in(0,+\infty]$ and $K,L \subset \mathbb{R}^{n}$ be convex bodies with $o\in int K\cap int L$ such that $K^{o,p}$ and $L^{o,p}$ have opposite barycenters.Then
$$C^{n}_{2n}\Big|conv\Big(K^{o,p}\cup L^{o,p}\Big)\Big|\cdot\Big|\Big[(K^{o,p})^{o}-(L^{o,p})^{o}\Big]^o\Big|\geq |K^{o,p}||L^{o,p}|,$$
where $C^{n}_{2n}=\frac{(2n)!}{(n!)^2}$ is the combination number formula.
\end{thm}

The ideas and techniques of Berndtsson, Mastrantonis and Rubinstein\cite{Berndtsson-Mastrantonis-Rubinstein-2023}, Meyer and Werner\cite{meyer-werner-1998}, Meyer and Pajor \cite{meyer-pajor-1990} and Alonso at al \cite{Alonso-2013, Alonso-2016, Alonso-2019} play a critical role throughout this paper. It would be impossible to overstate our reliance on their work.

\section{p Polar Body of Shawdow System}

For quick reference, we list some notations and definitions that will be useful in all the paper. See \cite{schneider-2014} for additional detail.
The origin, unit sphere, and closed unit ball in $n$-dimensional Euclidean space $\mathbb{R}^n$ are denoted by $o$, $S^{n-1}$ and $B^n_2$. If $u\in S^{n-1}$, then $u^\perp$ is the hyperplane containing $o$ and orthogonal $u$. For any convex body $K\subset\mathbb{R}^n$, we will denote by $I_K^\infty$ the convex characteristic function of $K$ which is the function defined by $I_K^\infty(x)=0$ if $x\in K$ and $I_K^\infty(x)=+\infty$ if $x\notin K$.

In addition to its denoting absolute value, we shall use $|\cdot|$ to denote the standard Euclidean norm on $\mathbb{R}^n$ and to denote $k$-dimensional volume in $\mathbb{R}^n$, $k\leq n$.

In this section we will prove the following theorems.

\begin{thm}\label{thm-p-support function is convex1}
 Let $p\in(0,\infty],v\in S^{n-1}$ and $K_t$ be a parallel movement of a convex body $K \subset \mathbb{R}^{n}$ along the direction $v$. Then the function $t\mapsto h_{p,K_t}$ is convex, i.e.,for any $x'\in v^\perp\ , s, t_i\in\mathbb{R}$, $i=1, 2$, we have
$$\frac{1}{2}h_{p,K_{t_{1}}}(x'+sv)+\frac{1}{2}h_{p,K_{t_{2}}}(x'+sv) \geq h_{p,K_{\frac{t_{1}+t_{2}}{2}}}(x'+sv).$$
\end{thm}

\begin{thm}\label{thm-main thm01}
Let $p\in(0,+\infty]$, $v\in S^{n-1}$ and $K_t$ be a parallel chord movement of a convex body $K$ along the direction $v$. Then for all $s, t_i\in\mathbb{R}$, $i=1,2$
\begin{align}\label{slice vulume inequality for p polar body of ti}
|K_{\frac{t_1+t_2}{2}}^{o,p}(s)|^{\frac{1}{n-1}}\geq \frac{1}{2}|K_{t_1}^{o,p}(s)|^{\frac{1}{n-1}}+\frac{1}{2}|K_{t_2}^{o,p}(s)|^{\frac{1}{n-1}}.
\end{align}
\end{thm}

\begin{lem}\label{lem kops=Kop-s}\cite[Corollary 5.16]{Berndtsson-Mastrantonis-Rubinstein-2023}
Let $p\in(0,+\infty]$, $v\in S^{n-1}$ and $K$ be an origin symmetric convex body. Then for all $s\in\mathbb{R}$, then $$K^{o,p}(s)=K^{o,p}(-s).$$
\end{lem}

\begin{lem}\label{lem K1s=Ks}
Let $p\in(0,+\infty]$, $v\in S^{n-1}$ and $K_t$ be a parallel chord movement of a convex body $K$ with speed function
$$\beta(\cdot):=-\Big(f_v(\cdot)+g_v(\cdot)\Big)$$
along the direction $v$. Then for all $s\in\mathbb{R}$,
$$K^{o,p}(s)=K_1^{o,p}(s).$$
\end{lem}

\begin{proof}
Since for any $v\in S^{n-1}$,  $$K=\{x'+sv|g_{v}(x')\leq s \leq f_{v}(x'),x'\in K|v^{\bot}\},$$
and from the definition of parallel movement (\ref{parallel chord movement}), we have
$$K_{1}=\{x'+sv|-f_{v}(x')\leq s \leq -g_{v}(x'),x'\in K|v^{\bot}\}.$$
Hence, for any $x'\in v^\perp$ and $s\in\mathbb{R}$,
\begin{align}\label{hpK1=hpK}
h_{p,K_{1}}(x'+sv)&=\frac{1}{p}\log\int_{K_{1}}e^{p\langle z'+z_{n}v,x'+sv \rangle}\frac{dz'dz_{n}}{|K_{1}|}\notag\\
&=\frac{1}{p}\log\int_{K|v^{\bot}}\int^{-g_{v}(x')}_{-f_{v}(x')}e^{p\langle z',x' \rangle}e^{psz_{n}}\frac{dz'dz_{n}}{|K|}\notag\\
&=\frac{1}{p}\log\int_{K|v^{\bot}}e^{p\langle z',x' \rangle}\frac{1}{ps}[e^{ps(-g_{v}(z'))}-e^{ps(-f_{v}(z'))}]\frac{dz'}{|K|}\notag\\
&=\frac{1}{p}\log\int_{K}e^{p\langle z'+z_{n}v,x'-sv \rangle}\frac{dz'dz_{n}}{|K|}\notag\\
&=h_{p,K}(x'-sv).
\end{align}
From (\ref{norm of polar body of K}) and (\ref{hpK1=hpK}), we get
\begin{align*}
\|x'+sv\|_{K^{o,p}_{1}}
=&\Big(\frac{1}{(n-1)!}\int^{\infty}_{0}r^{n-1}e^{-h_{p,K_1}(r(x'+sv))}dr\Big)^{-\frac{1}{n}}\\
=&\Big(\frac{1}{(n-1)!}\int^{\infty}_{0}r^{n-1}e^{-h_{p,K}(r(x'-sv))}dr\Big)^{-\frac{1}{n}}\\
=&\|x'-sv\|_{K^{o,p}}.
\end{align*}
Therefore, $x'+sv\in K^{o,p}_{1}$ if and only if $x'-sv\in K^{o,p}$ , which means that $$K^{o,p}(-s)=K^{o,p}_{1}(s).$$
\end{proof}

\begin{lem}\label{Lem-p support function inequlity}
 Let $p\in(0,\infty],v\in S^{n-1}$ and $K_t$ be a parallel movement of a convex body $K \subset \mathbb{R}^{n}$ along the direction $v$. Then for any $x',y'\in v^\perp, s\in\mathbb{R}$  and $ a,b,c>0 $ with $\frac{2}{a}=\frac{1}{b}+\frac{1}{c}$, we have
$$\frac{c}{b+c}h_{p,K_{t_{1}}}(bx'+bsv)+\frac{b}{b+c}h_{p,K_{t_{2}}}(cy'+csv) \geq h_{p,K_{\frac{t_{1}+t_{2}}{2}}}(a\frac{x'+y'}{2}+asv)$$
\end{lem}

\begin{proof}
For any  $v \in S^{n-1}$, there are convex functions $ g_{v}(\cdot),-f_{v}(\cdot): K|v^\perp\ \rightarrow \mathbb{R}$, such that
$$K=\{z'+rv: g_{v}(z')\leq r \leq f_{v}(z'), z'\in K|v^\perp\}.$$

Let $m_i:=\frac{1}{2}(g_{v}+f_{v})+\beta t_i, i=1,2$, where $\beta$ is the speed function of $K_t$ and $f_{v}:=f_{v}(z'),g_{v}:=g_{v}(z')$.

From $(1.2)$, Fubini's Theorem and $(1.8)$, we have

\begin{align}\label{Lem-p support function inequlity001}
&h_{p,K_{t_{1}}}(bx'+bsv)\notag\\
=&\frac{1}{p}\log\int_{K_{t_{1}}}e^{p\langle z,bx'+bsv \rangle}\frac{dz}{|K_{t_{1}}|}\nonumber\\
=&\frac{1}{p}\log\int_{K|v^\perp\ }\int^{f_{v}+\beta t_{1}}_{g_{v}+\beta t_{1}}e^{pb\langle z',x' \rangle}e^{pbsr}\frac{drdz'}{|K|}\nonumber\\
=&\frac{1}{p}\log\int_{K|v^\perp\ }e^{pb\langle z',x' \rangle}\frac{1}{pbs}\big( e^{pbs\frac{f_{v}-g_{v}}{2}}-e^{-pbs\frac{f_{v}-g_{v}}{2}}\big)\cdot e^{pbsm_1}dz'\nonumber\\
=&\frac{1}{p}\log\int_{K|v^\perp\ }e^{pb\langle z',x' \rangle}\cdot J(z',b)\cdot e^{pbsm_1(z')}dz'
\end{align}
where
\begin{align*}
J(z',b)=\frac{1}{pbs}(e^{pbs\frac{f_{v}-g_{v}}{2}}-e^{-pbs\frac{f_{v}-g_{v}}{2}}).
\end{align*}

In the similar way we have
\begin{align}\label{Lem-p support function inequlity002}
h_{p,K_{t_{2}}}(cy'+csv)=\frac{1}{p}\log\int_{K|v^\perp\ }e^{pc\langle z',y' \rangle}\cdot J(z',c)\cdot e^{pbsm_2(z')} dz',
\end{align}
and
\begin{align}\label{Lem-p support function inequlity003}
&h_{p,K_{\frac{t_{1}+t_{2}}{2}}}(a\frac{x'+y'}{2}+asv)\notag\\
=&\frac{1}{p}\log\int_{K|v^\perp\ }e^{p\frac{a}{2}\langle z',x'+y'\rangle}\cdot J(z',a)\cdot e^{pas\frac{m_1+m_2}{2}}dz'.
\end{align}

From $(\ref{Lem-p support function inequlity001})$ and $(\ref{Lem-p support function inequlity002})$, Minkowski inequality and the fact that $\frac{2}{a}=\frac{1}{b}+\frac{1}{c} $, we have
\begin{align}\label{Lem-p support function inequlity004}
&\frac{c}{b+c}h_{p,K_{t_{1}}}(bx'+bsv)+\frac{b}{b+c}h_{p,K_{t_{2}}}(cy'+csv)\notag\\
=&\frac{1}{p}\log\Big[\big(\int_{K|v^\perp\ }e^{pb\langle z',x' \rangle}\cdot J(z',b) e^{pbsm_1}dz'\big)^{\frac{c}{b+c}}\notag\\
&\cdot\big(\int_{K|v^\perp\ }e^{pc\langle z',y' \rangle}\cdot J(z',c)\cdot e^{pbsm_2}dz'\big)^{\frac{b}{b+c}}\Big]\notag\\
\geq&\frac{1}{p}\log\int_{K|v^\perp\ }e^{p\frac{bc}{b+c}\langle z',x' \rangle}\cdot J(z',b)^{\frac{c}{b+c}}\cdot e^{\frac{pbcsm_1}{b+c}}\cdot e^{p\frac{bc}{b+c}\langle z',y' \rangle}\cdot J(z',c)^{\frac{b}{b+c}}\cdot e^{\frac{pbcsm_2}{b+c}}dz'\notag\\
=&\frac{1}{p}\log\int_{K|v^\perp\ }e^{p\frac{a}{2}\langle z',x'+y'\rangle}\cdot J(z',b)^{\frac{c}{b+c}}\cdot J(z',c)^{\frac{b}{b+c}}\cdot e^{pas\frac{m_1+m_2}{2}(z')}dz'.
\end{align}

Since $J(z',t)=\frac{2}{pst}\sinh(\frac{f_v-g_v}{2}pst)$ is log-convex in $t$, i.e.
$$J(z',b)^{\frac{c}{b+c}}\cdot J(z',c)^{\frac{b}{b+c}}\geq J(z',a),$$
from (\ref{Lem-p support function inequlity003}) and (\ref{Lem-p support function inequlity004}), we have
\begin{align*}
&\frac{c}{b+c}h_{p,K_{t_{1}}}(bx'+bsv)+\frac{b}{b+c}h_{p,K_{t_{2}}}(cy'+csv)\\
&\geq\frac{1}{p}\log\int_{K|v^\perp\ }e^{pa\langle z',\frac{x'+y'}{2} \rangle}\cdot J(z',a)\cdot e^{pas\frac{m_1+m_2}{2}}dz'\\
&=h_{p,K_{\frac{t_{1}+t_{2}}{2}}}(a\frac{x'+y'}{2}+asv)\\
\end{align*}
\end{proof}

\emph{Proof of Theorem \ref{thm-p-support function is convex1}.}
The direct conclusion of the Lemma \ref{Lem-p support function inequlity} is that
the $L_p$-support function of the shadow system $K_t$ is convex with respect to parameter $t$.
$\square$

The following well known consequence of Keith Ball will be needed.
\begin{lem}\label{keith-ball-1988}\cite[Theorem 5]{keith-ball-1988}
Let $F, G, H: (0,+\infty)\rightarrow [0,+\infty)$ be measurable functions, not almost everywhere $0$, with
$$H(r)\geq F(t)^{\frac{s}{t+s}}\cdot G(s)^{\frac{t}{t+s}},\ \ \ \frac{2}{r}=\frac{1}{s}+\frac{1}{t}.$$
Then, for $q\geq1$,
$$2\Big(\int_0^\infty r^{q-1}H(r)dr\Big)^{-\frac{1}{q}}\leq \Big(\int_0^\infty t^{q-1}F(t)dt\Big)^{-\frac{1}{q}}+\Big(\int_0^\infty s^{q-1}G(s)ds\Big)^{-\frac{1}{q}}.$$
\end{lem}

\begin{lem}\label{lem-section of Kt1-Kt2-subset Kt12}
Let $p\in(0,\infty],v\in S^{n-1}$ and $K_t$ be a parallel chord movement of a convex body $K$ along the direction $v$. Then for any $s,t_i\in \mathbb{R},i=1,2.$
$$\frac{1}{2}K^{o,p}_{t_{1}}(s)+\frac{1}{2}K^{o,p}_{t_{2}}(s)\subset K^{o,p}_{\frac{t_{1}+t_{2}}{2}}(s).$$
\end{lem}
\begin{proof}
For any $x'+sv\in K^{o,p}_{t_{1}} ,y'+sv\in K^{o,p}_{t_{2}} $, and $a,b,c\in\mathbb{R}$ with $\frac{2}{a}=\frac{1}{b}+\frac{1}{c}$, from Lemma \ref{Lem-p support function inequlity}, we have
\begin{align}\label{exp-hpKt1t2--leq---}
e^{-h_{p,K_{\frac{t_{1}+t_{2}}{2}}}(a\frac{x'+y'}{2}+asv)}
\geq\Big(e^{-h_{p,K_{t_{1}}}(bx'+bsv)}\Big)^{\frac{c}{b+c}}
\cdot\Big(e^{-h_{p,K_{t_{2}}}(cy'+csv)}\Big)^{\frac{b}{b+c}}.\notag\\
\end{align}
From (\ref{norm of polar body of K}), (\ref{exp-hpKt1t2--leq---}) and Lemma \ref{keith-ball-1988}, we get
\begin{align*}
&\|\frac{x'+y'}{2}+sv \|_{K^{o,p}_{\frac{t_{1}+t_{2}}{2}}}\\
=&\Big(\frac{1}{(n-1)!}\int^{\infty}_{0}a^{n-1}e^{-h_{p,K_{\frac{t_{1}+t_{2}}{2}}}(a\frac{x'+y'}{2}+asv)}da \Big)^{-\frac{1}{n}}\\
\leq& \frac{1}{2}\Big(\frac{1}{(n-1)!}\int^{\infty}_{0}b^{n-1}e^{-h_{p,K_{t_{1}}}(bx'+bsv)}db\Big)^{-\frac{1}{n}}\\
&+\frac{1}{2}\Big(\frac{1}{(n-1)!}\int^{\infty}_{0}c^{n-1}e^{-h_{p,K_{t_{2}}}(cy'+csv)}dc\Big)^{-\frac{1}{n}}\\
=&\frac{1}{2}\|x'+sv\|_{K^{o,p}_{t_{1}}}+\frac{1}{2}\|y'+sv\|_{K^{o,p}_{t_{2}}}\leq1.
\end{align*}
Hence  $$\frac{x'+y'}{2}+sv \in  K^{o,p}_{\frac{t_{1}+t_{2}}{2}},$$
which means that
$$\frac{1}{2}K^{o,p}_{t_{1}}(s)+\frac{1}{2}K^{o,p}_{t_{2}}(s)\subset K^{o,p}_{\frac{t_{1}+t_{2}}{2}}(s).$$
\end{proof}

\emph{Proof of Theorem \ref{thm-main thm01}.}
It is obvious that the inequality (\ref{slice vulume inequality for p polar body of ti}) follows from the Brunn-Minkowski inequality and
Lemma \ref{lem-section of Kt1-Kt2-subset Kt12}.
$\square$

\begin{corr}
Let $p\in(0,+\infty]$, $v\in S^{n-1}$ and $K_t$ be a parallel chord movement of a convex body $K$ with speed function $\beta(\cdot)=-\big(g_v(\cdot)+ f_v(\cdot)\big)$ along the direction $v$. Then for all $s\in\mathbb{R}$,
\begin{align}\label{slice vulume inequality for p polar body of 0--1}
\frac{1}{2}K^{o,p}(s)+\frac{1}{2}K_{1}^{o,p}(s)\subset K_{\frac{1}{2}}^{o,p}(s)=(S_vK)^{o,p}(s).
\end{align}
In particular, for $K=-K$,
\begin{align}\label{steiner symmetrical inequality}
|(S_vK)^{o,p}|\geq |K^{o,p}|.
\end{align}
\end{corr}
\begin{proof}
Taking $t_1=0$ and $t_2=1$ in Lemma \ref{lem-section of Kt1-Kt2-subset Kt12} and Theorem \ref{thm-main thm01}, it is easily to obtain (\ref{slice vulume inequality for p polar body of 0--1}) and (\ref{steiner symmetrical inequality}) follows from (\ref{slice vulume inequality for p polar body of 0--1}), Lemma \ref{lem K1s=Ks} and Lemma \ref{lem kops=Kop-s}.
\end{proof}

\begin{rem}
We should point out that the above inequality (\ref{steiner symmetrical inequality}) has been proven in \cite{Berndtsson-Mastrantonis-Rubinstein-2023}.
\end{rem}

The following $L_p$ Blaschke-Santal\'{o} inequality follows from (\ref{steiner symmetrical inequality}), (\ref{volume of p-polar body}), (\ref{p-mahler volume}) and the fact that repeated Steiner symmetrizations converge to a dilated Euclidean standard ball.

\begin{corr}\cite{Berndtsson-Mastrantonis-Rubinstein-2023}
Let $p\in(0,+\infty]$. For a symmetric convex body $K\subset \mathbb{R}^n$,
\begin{align}\label{p-blaschke santalo inequality}
\mathcal{M}_p(K)\leq \mathcal{M}_p(B_2^n),
\end{align}
where $B_2^n=\{x\in\mathbb{R}^n:|x_1|^2+\cdots+|x_n|^2\leq1\}$.
\end{corr}

\begin{rem}
We want to point out that the above inequality (\ref{p-blaschke santalo inequality}) in \cite{Berndtsson-Mastrantonis-Rubinstein-2023} is stated for symmetric bodies and we can get a different upper bound for $\mathcal{M}_p(K)$ of convex body $K$, which may not be symmetric.
\end{rem}

We need the following Lemma which is classical and can be found in \cite[Lemma 3]{meyer-werner-1998} or in \cite[Lemma 2.5]{fradelizi-arxiv-202307}.

\begin{lem}\label{lem K-z integral volume fomula}
Let $K \subset \mathbb{R}^{n}$ be a convex body with $o, z\in int K$.
Then
$$|(K-z)^o|=\int_{K^o}\frac{dx}{(1-\langle z,x\rangle)^{n+1}}.$$
\end{lem}

\begin{lem}\label{lem Ko-Kop-Ko}\cite[Remark 3.14]{Berndtsson-Mastrantonis-Rubinstein-2023}
Let $p\in(0,\infty]$ and $K \subset \mathbb{R}^{n}$ be a convex body with $bar(K)=o$.
Then
$$K^o\subset K^{o,p}\subset  \frac{(1+p)^{1+\frac{1}{p}}}{p}K^o.$$
\end{lem}

\begin{prop}\label{thm-sharp-blaschke-santalo-inequality}
Let $p\in(0,\infty]$ and $K \subset \mathbb{R}^{n}$ be a convex body with $bar(K)=o$. Then
$$\mathcal{M}_p(K)\leq n!\Big(\frac{(1+p)^{1+\frac{1}{p}}}{p}\Big)^n|B_2^n|^2.$$
\end{prop}

\begin{proof}
Let $z\in intK^{o,p}$. From Lemma \ref{lem K-z integral volume fomula}, Lemma \ref{lem Ko-Kop-Ko} and applying Jensen's inequality to the convex
function $\varphi(x)=(1-\langle z,x\rangle)^{-(n+1)}$ on $K$, we get
\begin{align}\label{lem Ko-Kop-Ko 001}
|(K^{o,p}-z)^o|&=\int_{(K^{o,p})^o}\frac{dx}{(1-\langle z,x \rangle)^{n+1}}\nonumber\\
&\geq\int_{\frac{p}{(1+p)^{1+\frac{1}{p}}}K}\frac{dx}{(1-\langle z,x \rangle)^{n+1}}\nonumber\\
&\geq
\Big(\frac{p}{(1+p)^{1+\frac{1}{p}}}\Big)^n|K|\varphi\Big(\int_{\frac{p}{(1+p)^{1+\frac{1}{p}}}K}x\frac{dx}{|\frac{p}{(1+p)^{1+\frac{1}{p}}}K|}\Big)\nonumber\\
&=\Big(\frac{p}{(1+p)^{1+\frac{1}{p}}}\Big)^n|K|,
\end{align}
where the barycenter of $\Big(\frac{p}{(1+p)^{1+\frac{1}{p}}}K\Big)$ is $o$.

Let $z=San(K^{o,p})$ be the Santal\'{o} point $K^{o,p}$. Using (\ref{lem Ko-Kop-Ko 001}) and the classical Blaschke Santal\'{o} inequality,
we can obtain
\begin{align*}
|K||K^{o,p}|&\leq\Big(\frac{(1+p)^{1+\frac{1}{p}}}{p}\Big)^n|\big(K^{o,p}-San(K^{o,p})\big)^o||K^{o,p}|\\
&\leq \Big(\frac{(1+p)^{1+\frac{1}{p}}}{p}\Big)^n|B_2^n|^2,
\end{align*}
which shows that
$$\mathcal{M}_p(K)\leq n!\Big(\frac{(1+p)^{1+\frac{1}{p}}}{p}\Big)^n|B_2^n|^2.$$
\end{proof}

\section{The reverse Rogers-Shephard type inequality for p polar bodies}
In this section we will discuss reverse Rogers-Shephard type inequality for $L^p$-polar bodies.
We will need the following useful Lemmas.

\begin{lem}\label{lem-0 in K=0 in kop}\cite[Lemma 3.1]{Berndtsson-Mastrantonis-Rubinstein-2023}
 Let $p\in(0,\infty]$ and $K$ be a convex body. Then
 \begin{align*}
o\in int K \ \ \ \text{if\ and\ only\ if}\ \ \  o\in intK^{o,p}.
\end{align*}
\end{lem}

\begin{lem}\cite[Theorem 1.6.3]{schneider-2014}\label{lem-conv(KL)}
Let $K,L \subset \mathbb{R}^{n}$ be convex bodies with $o\in int(K\cap L)$. Then
$$\big(K^o\cap L^o\big)^o=conv(K\cup L).$$
\end{lem}

\begin{lem}\label{lem-legendre transform}
Let $p\in(0,+\infty]$ and $K,L \subset \mathbb{R}^{n}$ be convex bodies with $o\in int(K\cap L)$.
Then, for any $z\in\mathbb{R}^n$,
$$\inf_{z=x+y}(\|x\|_{K^{o,p}}+\|y\|_{K^{o,p}})=\|z\|_{conv(K^{o,p}\cup L^{o,p})}.$$
\end{lem}

\begin{proof}
For any convex function $f$ on $\mathbb{R}^n$, its \textit{Legendre transform} $\mathcal{L}f$ is defined by
$$\mathcal{L}f(x)=\sup_{y\in\mathbb{R}^n}\{\langle x,y \rangle-f(y)\}.$$

In fact, for any convex body $K,L\subset\mathbb{R}^n$ with $o\in int(K\cap L)$, we have
\begin{align}\label{lem-legendre transform001}
&\mathcal{L}\Big(\inf_{\cdot=x+y}(\|x\|_K+\|y\|_L)\Big)(z)\notag\\
=&\sup_{w\in\mathbb{R}^n}\Big(\langle w,z\rangle-\inf_{w=x+y}(\|x\|_K+\|y\|_L)\Big)\notag\\
=&\sup_{w\in\mathbb{R}^n}\Big(\langle w,z\rangle-\inf_{w=x+y}(h_{K^o}(x)+h_{L^o}(y))\Big)\notag\\
=&\sup_{w\in\mathbb{R}^n}\sup_{w=x+y}\Big(\langle x,z\rangle-h_{K^o}(x)+\langle y,z\rangle-h_{L^o}(y)\Big)\notag\\
=&I_{K^o}^\infty(z)+I_{L^o}^\infty(z).
\end{align}

Hence, from Lemma \ref{lem-conv(KL)},
\begin{align}\label{lem-legendre transform002}
&\inf_{z=x+y}(\|x\|_K+\|y\|_L)\notag\\
=&\mathcal{L}\Big(I_{K^o}^\infty(\cdot)+I_{L^o}^\infty(\cdot)\Big)(z)\notag\\
=&\sup_{w\in\mathbb{R}^n}\Big(\langle w,z\rangle-I_{K^o}^\infty(w)-I_{L^o}^\infty(w)\Big)\notag\\
=&h_{K^o\cap L^o}(z)=\|z\|_{\big(K^o\cap L^o\big)^o}\notag\\
=&\|z\|_{conv(K\cup L)}.
\end{align}

The conclusion follows from Lemma \ref{lem-0 in K=0 in kop} and (\ref{lem-legendre transform002}).
\end{proof}

\begin{lem}\label{lem-inequality of log concave function}\cite{milman-2000}
Let $\mu$ be a probability measure on $\mathbb{R}^n$ and let $\phi: \mathbb{R}^n\rightarrow\mathbb{R}$
be a non-negative log-concave function with finite, positive integral. Then
$$\int_{\mathbb{R}^n}\phi(x)d\mu(x)\leq\Big(\int_{\mathbb{R}^n}x\frac{\phi(x)}{\int_{\mathbb{R}^n}\phi(y)d\mu(y)}d\mu(x)\Big).$$
\end{lem}

\begin{lem}\label{Lem-norm of K cha cheng L}
Let $K,L \subset \mathbb{R}^{n}$ be convex bodies. Then, for any $x,y\in\mathbb{R}^n$,
$$\max\{\|x\|_K, \|y\|_L\}=\|(x,y)\|_{K\times L}.$$
\end{lem}

In fact, from the definition of norm, we have
\begin{align*}
\|(x,y)\|_{K\times L}
=&\inf\{\lambda:(x,y)\in\lambda(K\times L)\}\notag\\
=&\max\{\inf\{\lambda:x\in\lambda K\}, \inf\{\lambda:y\in\lambda L\}\}\notag\\
=&\max\{\|x\|_K, \|y\|_L\}.
\end{align*}

Let $K,L \subset \mathbb{R}^{n}$ be convex bodies with $o\in int(K\cap L)$. For simplification, we introduce the following notations.
\begin{align}\label{f(x)--g(x)}
f(x)=e^{-\|x\|_{K^{o,p}}}\ \ \ \text{and}\ \ \ g(y)=e^{-\|y\|_{L^{o,p}}}.
\end{align}

\begin{align}\label{Ct}
C_{t}=\{(x,y)\in \mathbb{R}^{2n}:f(x)g(-y)\geq t\}
\end{align}

\begin{align}\label{--Ct}
\tilde{C}_{t}=\{(u,v)\in \mathbb{R}^{2n}:f(\frac{u+v}{\sqrt{2}})g(-\frac{u-v}{\sqrt{2}})\geq t\}.
\end{align}

Let $M_t$ be the projection of $\tilde{C}_t$ onto the subspace
$$H=span\{e_{n+1},\cdots,e_{2n}\}=\mathbb{R}^n,$$
thus
\begin{align}\label{Mt}
M_{t}=\{(o,v)\in \mathbb{R}^{2n}:\max_{u\in\mathbb{R}^n} f(\frac{u+v}{\sqrt{2}})g(-\frac{u-v}{\sqrt{2}})\geq t\}.
\end{align}

\begin{thm}\label{thm-reverse--R-S-type-inequality}
Let $p\in(0,+\infty]$ and $K,L \subset \mathbb{R}^{n}$ be convex bodies with $o\in int K\cap int L$ such that $K^{o,p}$ and $L^{o,p}$ have opposite barycenters.Then
$$C^{n}_{2n}\Big|conv\Big(K^{o,p}\cup L^{o,p}\Big)\Big|\cdot\Big|\Big((K^{o,p})^{o}-(L^{o,p})^{o}\Big)^o\Big|\geq |K^{o,p}||L^{o,p}|,$$
where $C^{n}_{2n}=\frac{(2n)!}{(n!)^2}$ is the combination number formula.
\end{thm}

\begin{proof}
From the definitions (\ref{--Ct}) and (\ref{Mt}), the volume of $\tilde{C}_t$ can be expressed as an integral,
\begin{align}\label{volume --Ct001}
|\tilde{C}_t|=\int_{M_t}\int_{\tilde{C}_t\cap\big((0,v)+H^{\perp}\big)}dudv=\int_{M_t}\phi_t(v)dv,
\end{align}
where $\phi_t(v)=|\tilde{C}_t\cap\big((o,v)+H^{\perp}\big)|$ is the volume of the section $\tilde{C}_t\cap\big((o,v)+H^{\perp}\big)$.

Since $\phi_t(v)$ is a log-concave function and has support on $M_t$, from (\ref{volume --Ct001}) Lemma \ref{lem-inequality of log concave function}, we have

\begin{align}\label{volume inequality --Ct001}
|\tilde{C}_t|=&|M_t|\int_{M_t}\phi_t(v)\frac{dv}{|M_t|}
\leq|M_t|\phi_t\Big(\frac{\int_{M_t}v\phi_t(v)\frac{dv}{|M_t|}}{\int_{M_t}\phi_t(v)\frac{dv}{|M_t|}}\Big)\notag\\
=&|M_t|\phi_t\Big(\frac{1}{|\tilde{C}_t|}\int_{M_t}\int_{\tilde{C}_t\cap\big((0,v)+H^{\perp}\big)}vdudv\Big)\notag\\
=&|M_t|\phi_t\Big(\frac{1}{|\tilde{C}_t|}\int_{\tilde{C}_t}vdudv\Big).
\end{align}

By using variable substitution $u=\frac{x+y}{\sqrt{2}}$, $v=\frac{x-y}{\sqrt{2}}$, we have
\begin{align}\label{volume inequality --Ct002}
\int_{\tilde{C}_t}vdudv
=&\frac{1}{\sqrt{2}}\int_{\{(x,y):f(x)g(-y)\geq t\}}x-ydxdy\notag\\
=&\frac{1}{\sqrt{2}}\Big(\int_{C_t}xdxdy-\int_{C_t}ydxdy\Big).
\end{align}

Since $\|f\|_\infty=\|g\|_\infty=1$, we have
\begin{align}\label{volume inequality --Ct003}
\int_{C_t}xdxdy=&\int_{\{(x,y):f(x)g(-y)\geq t\}}xdydx\notag\\
=&\int_{\{x:f(x)\geq t\}}\int_{\{y:g(-y)\geq \frac{t}{f(x)}\}}xdydx\notag\\
=&\int_{\{x:\|x\|_{K^{o,p}}\leq -\log t\}}\int_{\{y:\|-y\|_{L^{o,p}}\leq -\log t-\|x\|_{K^{o,p}}\}}xdydx\notag\\
=&|-L^{o,p}|\int_{(-\log t)K^{o,p}}x(-\log t-\|x\|_{K^{o,p}})^ndx\notag\\
=&|L^{o,p}|(-\log t)^{2n+1}\int_{K^{o,p}}z(1-\|z\|_{K^{o,p}})^ndz\notag\\
=&(-\log t)^{2n+1}|L^{o,p}|\int_{K^{o,p}}z\int_{\|z\|_{K^{o,p}}}^1n(1-s)^ndsdz\notag\\
=&(-\log t)^{2n+1}|L^{o,p}|n\int_0^1(1-s)^n\int_{sK^{o,p}}zdzds\notag\\
=&(-\log t)^{2n+1}|K^{o,p}||L^{o,p}|n\int_0^1(1-s)^ns^{n+1}ds\int_{K^{o,p}}z\frac{dz}{|K^{o,p}|}\notag\\
=&n(-\log t)^{2n+1}\beta(n,n+2)|K^{o,p}||L^{o,p}|\cdot bar(K^{o,p}),
\end{align}
where $\beta(n,n+2)$ is Euler integral. In the similar way, we have

\begin{align}\label{volume inequality --Ct004}
\int_{C_t}ydxdy=n(-\log t)^{2n+1}\beta(n,n+2)|K^{o,p}||L^{o,p}|\cdot bar(L^{o,p}).
\end{align}

From (\ref{volume inequality --Ct001}), (\ref{volume inequality --Ct002}), (\ref{volume inequality --Ct003}), (\ref{volume inequality --Ct004}) and the fact that $bar(K^{o,p})=-bar(L^{o,p})$, we get
\begin{align}\label{volume inequality --Ct005}
|\tilde{C}_t|\leq |M_t|\phi_t(o)= |M_t|\cdot|\tilde{C}_t\cap H^\perp|.
\end{align}

On the one hand,
\begin{align}\label{integral 0-1--Ct-one hand}
\int_0^1\tilde{C}_{t}dt
=&\int_0^1\int_{\{(u,v)\in \mathbb{R}^{2n}:f(\frac{u+v}{\sqrt{2}})g(-\frac{u-v}{\sqrt{2}})\geq t\}}dudvdt\notag\\
=&\int_{\mathbb{R}^{2n}}f(\frac{u+v}{\sqrt{2}})g(-\frac{u-v}{\sqrt{2}})dudv\notag\\
=&\int_{\mathbb{R}^{2n}}f(x)g(-y)dxdy\notag\\
=&\int_{\mathbb{R}^{2n}}e^{-\|x\|_{K^{o,p}}}e^{-\|y\|_{L^{o,p}}}dxdy\notag\\
=&(n!)^2|K^{o,p}||L^{o,p}|.
\end{align}

On the other hand, from (\ref{volume inequality --Ct005}), Lemma \ref{lem-legendre transform} and Lemma \ref{Lem-norm of K cha cheng L}, we have
\begin{align}\label{integral 0-1--Ct-other hand}
&\int_0^1\tilde{C}_{t}dt
\leq\int_0^1|M_t|\cdot|\tilde{C}_t\cap H^\perp|dt\notag\\
=&\int_0^1|\{(o,v)\in \mathbb{R}^{2n}:\max_{\bar{u}\in\mathbb{R}^n} f(\frac{\bar{u}+v}{\sqrt{2}})g(-\frac{\bar{u}-v}{\sqrt{2}})\geq t\}|\notag\\
&\cdot|\{u\in \mathbb{R}^n: f(\frac{u}{\sqrt{2}})g(-\frac{u}{\sqrt{2}})\geq t\}|dt\notag\\
=&\int_{\mathbb{R}^{2n}}\int_0^{\min\Big\{\max_{\bar{u}\in\mathbb{R}^n} f(\frac{\bar{u}+v}{\sqrt{2}})g(-\frac{\bar{u}-v}{\sqrt{2}}), f(\frac{u}{\sqrt{2}})g(-\frac{u}{\sqrt{2}})\Big\}}dtdudv\notag\\
=&\int_{\mathbb{R}^{2n}}\min\{\max_{\tilde{u}\in\mathbb{R}^n}e^{-\|\frac{\tilde{u}+v}{\sqrt{2}}\|_{K^{o,p}}-\|\frac{v-\tilde{u}}{\sqrt{2}}\|_{L^{o,p}}},
e^{-\|\frac{u}{\sqrt{2}}\|_{K^{o,p}}-\|\frac{-u}{\sqrt{2}}\|_{L^{o,p}}}\}dudv\notag\\
=&\int_{\mathbb{R}^{2n}}\min\{e^{-\|\sqrt{2}v\|_{conv(K^{o,p}\cup L^{o,p})}},
e^{-h_{(K^{o,p})^{o}-(L^{o,p})^{o}}(\frac{u}{\sqrt{2}})}dudv\notag\\
=&\int_{\mathbb{R}^{2n}}e^{-\max\{\|v\|_{conv(K^{o,p}\cup L^{o,p})},\|u\|_{((K^{o,p})^{o}-(L^{o,p})^{o})^o}\}}dudv\notag\\
=&\int_{\mathbb{R}^{2n}}e^{-\|(v,u)\|_{conv(K^{o,p}\cup L^{o,p})\times((K^{o,p})^{o}-(L^{o,p})^{o})^o}}dudv\notag\\
=&(2n)!|conv(K^{o,p}\cup L^{o,p})\times((K^{o,p})^{o}-(L^{o,p})^{o})^o|\notag\\
=&(2n)!|conv(K^{o,p}\cup L^{o,p})|\cdot|((K^{o,p})^{o}-(L^{o,p})^{o})^o|.
\end{align}

The conclusion follows from (\ref{integral 0-1--Ct-one hand}) and  (\ref{integral 0-1--Ct-other hand}).
\end{proof}

\begin{rem}
Using the following inequality due to Rogers and Shephard \cite[P273, Theorem 1]{rogers-1958} instead of (\ref{volume inequality --Ct005}),
\begin{align*}
|\tilde{C}_t|\geq \frac{(n!)^2}{(2n)!}|M_t|\cdot|\tilde{C}_t\cap H^\perp|,
\end{align*}
we can obtain the reverse of Theorem \ref{thm-reverse--R-S-type-inequality} as follows,
$$\Big|conv\Big(K^{o,p}\cup L^{o,p}\Big)\Big|\cdot\Big|\Big((K^{o,p})^{o}-(L^{o,p})^{o}\Big)^o\Big|\leq |K^{o,p}||L^{o,p}|.$$

\end{rem}

\bibliographystyle{Plain}

\begin{thebibliography}{10}

\bibitem{Alonso-2013}
D. Alonso-Guti\'{e}rrez, C. H. Jim\'{e}nez, R. Villa,
\textit{Brunn-Minkowski and Zhang inequalities for convolution bodies},
Adv. in Math. 238(2013), 50-69.

\bibitem{Alonso-2016}
D. Alonso-Guti\'{e}rrez, B. Gonz\'{a}lez Merino, C. H. Jim\'{e}nez, R. Villa,
\textit{Rogers-Shephard inequality for log-concave functions},
J. Func. Anal. 271(11)(2016), 3269-3299.

\bibitem{Alonso-2019}
D. Alonso-Guti\'{e}rrez,
\textit{A reverse Rogers-Shephard inequality for log-concave functions}, J. Geom. Anal. 29(2019), 299-315.

\bibitem{Artstein--2015}
S. Artstein-Avidan, K. Einhorn, D. I. Florentin and Y. Ostrover, \textit{On Godbersen's
conjecture}, Geom. Dedicata 178 (1) (2015), 337-350.

\bibitem{keith-ball-1988}
K. Ball, \textit{Logarithmically concave functions and sections of convex sets in $\mathbb{R}^n$}, Studia Math. 88(1)(1988), 69-84.

\bibitem{Berndtsson-2022}
B. Berndtsson, \textit{Bergman kernels for Paley-Wiener space and Nazarov's proof of the
Bourgain-Milman theorem}, Pure Appl. Math. Q. 18(2022), 395-409.

\bibitem{Berndtsson-Mastrantonis-Rubinstein-2023}
B. Berndtsson, V. Mastrantonis and Y.A. Rubinstein,
$L^{p}$-polarity, Mahler volumes, and the isotropic constant, (2023),arXiv:2304.14363.

\bibitem{Bianchini-2008}
C. Bianchini and A. Colesanti, \textit{A sharp Rogers and Shephard inequality for the pdifference body of planar convex bodies}, Proc. Amer. Math. Soc. 136 (2008),
2575-2582.

\bibitem{blaschke-1923}
W. Blaschke, \textit{Vorlesungen \"{u}ber Differentialgeometrie II: Affine Differentialgeometrie}. Springer-Verlag,
Berlin, 1923.

\bibitem{campi-2002}
S. Campi and P. Gronchi, \textit{The $L_p$-Busemann-Petty centroid inequality}, Adv. Math. 167 (2002), 128-141.

\bibitem{campi-2002-02}
S. Campi and P. Gronchi, \textit{On the reverse $L_p$-Busemann-Petty centroid inequality}, Mathematika 49 (2002), 1-11.

\bibitem{campi-2006}
S. Campi and P. Gronchi, \textit{On volume product inequalities for convex sets}, Proc. Amer. Math. Soc. 134 (2006), 2398-2402.

\bibitem{campi-2006-02}
S. Campi and P. Gronchi, \textit{Volume inequalities for $L_p$-zonotopes}, Mathematika 53 (2006), 71-80.

\bibitem{fradelizi-arxiv-202307}
M. Fradelizi, N. Gozlan, S. Sadovsky and S. Zugmeyer
\textit{Transport-entropy forms of direct and Converseblaschke-Santal'{o} inequalities},(2023,)arXiv:2307.04393.

\bibitem{fradelizi-meyer-2012}
M. Fradelizi, M. Meyer and A. Zvavitch, \textit{An application of shadow systems to Mahler's
conjecture}, Discrete Comput. Geom. 48 (2012), 721-734.


\bibitem{LYZ-2000}
E. Lutwak, D. Yang and G. Zhang,\textit{$L_p$ affine isoperimetric inequalities}, J. Differential
Geom. 56 (2000), 111-132.

\bibitem{lutwak-zhang-1997}
E. Lutwak and G. Zhang, \textit{Blaschke-Santalo inequalities}, J. Differential Geom. 47 (1997), 1-16.

\bibitem{meyer-1991}
M. Meyer, \textit{Convex bodies with minimal volume product in $\mathbb{R}^2$}, Monatsh. Math. 112 (1991), 297-301.

\bibitem{meyer-pajor-1989}
M. Meyer and A. Pajor, \textit{On Santal\'{o}'s inequality}. In Geometric Aspects of Functional Analysis (J. Lindenstrauss, V.D. Milman, eds), Lecture Notes in Math. 1376, pp.261-263, Springer, Berlin, 1989.

\bibitem{meyer-pajor-1990}
M. Meyer and A. Pajor, \textit{On the Blaschke-Santal\'{o} inequality}, Arch. Math. 55 (1990), 82-93.

\bibitem{meyer-reisner-2006}
M. Meyer and S. Reisner, \textit{Shadow systems and volumes of polar convex bodies}, Mathematika 53 (2006), 129-148.

\bibitem{meyer-reisner-2011}
M. Meyer and S. Reisner, \textit{On the volume product of polygons}, Abh. Math. Sem. Univ. Hamburg 81 (2011), 93-100.

\bibitem{meyer-werner-1998}
M. Meyer and E. Werner, \textit{The Santal\'{o}-regions of a convex body}. Trans. Am. Math. Soc., 350(11)(1998),4569-4591.

\bibitem{milman-2000}
V.M. Milman and A. Pajor, \textit{Entropy and Asymptotic Geometry of Non-Symmetric Convex Bodies}, Adv. Math. 152(2) (2000), 314-335.

\bibitem{rogers-1957}
C.A. Rogers and G.C. Shephard, \textit{The difference body of a convex body}, Arch. Math. 8 (1957), 220-233.

\bibitem{rogers-1958}
C.A. Rogers and G.C. Shephard, \textit{Convex bodies associated with a given convex body}, J. Lond. Math. Soc. 33 (1958), 270-281.

\bibitem{rogers-1958-02}
C.A. Rogers and G.C. Shephard, \textit{Some extremal problems for convex bodies}, Mathematika 5 (1958), 93-102.

\bibitem{santalo-1949}
L. A. Santal\'{o}, \textit{An affine invariant for convex bodies of $n$-dimensional space}. Portugal. Math., 8(1949),155-161.

\bibitem{schneider-2014}
R. Schneider, Convex Bodies: The Brunn-Minkowski Theory,
Encyclopedia Math. Appl., vol. 151, expanded edn. Cambridge University Press,
Cambridge, 2014.

\bibitem{shephard-1964}
G.C. Shephard, Shadow systems of convex sets, Israel J. Math. 2 (1964), 229-236.



\end{thebibliography}

\end{document}